\documentclass[11pt,leqno]{amsart}
\usepackage[margin=1.45in]{geometry}
\usepackage[colorlinks=true,pdfborder={0 0 0}]{hyperref}
\hypersetup{
	colorlinks=true,       
	linkcolor=red,          
	citecolor=blue,        
	filecolor=magenta,      
	urlcolor=cyan           
}
\usepackage{amsmath}
\usepackage{amsfonts,amssymb}
\usepackage{dsfont}
\usepackage{enumerate}
\usepackage{mathrsfs}
\usepackage{amsthm}
\usepackage{xcolor}
\usepackage{tikz}
\usepackage{textcomp}
\usepackage{ucs}
\usepackage{textcomp}
\usepackage{dsfont}
\theoremstyle{plain}
\newtheorem{theo}{Theorem}[section]
\newtheorem*{theo*}{Theorem}
\newtheorem{prop}[theo]{Proposition}
\newtheorem{lemm}[theo]{Lemma}

\newtheorem{defi}[theo]{Definition}
\theoremstyle{definition}

\DeclareMathOperator{\supp}{supp}

\DeclareMathOperator{\diff}{d}
\DeclareSymbolFont{pletters}{OT1}{cmr}{m}{sl}
\DeclareMathSymbol{s}{\mathalpha}{pletters}{`s}

\def\ba{\begin{align}}
	\def\bad{\begin{aligned}}
		\def\be{\begin{equation}}
			\def\ea{\end{align}}
		\def\ead{\end{aligned}}
	\def\ee{\end{equation}}

\def \D{\diff\! }

\def\dt{\diff \! t}
\def\ddt{\frac{{\rm d}}{{\rm d}t}}
\def\ds{\diff \! s}
\def\dx{\diff \! x}
\def\da{\diff \! a}
\def\dxa{\diff \! x \diff \! a}

\def\dy{\diff \! y}

\def\dd{\diff\! }
\def\dtau{\diff \! \tau}

\def\le{\leq}

\newcommand{\var}{\varepsilon}

\def \R{\mathbb{R}}

\numberwithin{equation}{section}

\title{Concentration in an advection-diffusion model with
  diffusion coefficient depending on the past trajectory}

\author{Cosmin Burtea} \address[C. Burtea]{Institut Mathématique
  de Jussieu, Université Paris Cité, Paris, France}

\email{cburtea@math.univ-paris-diderot.fr}

\author{Nicolas Meunier} \address[N. Meunier]{LaMME, UMR CNRS
  8071, Universit\'e d'\'Evry, Val d’Essonne, France}

\email{nicolas.meunier@univ-evry.fr}

\author{Cl\'ement Mouhot} \address[C. Mouhot]{DPMMS, University
  of Cambridge, Cambridge, UK} \email{c.mouhot@dpmms.cam.ac.uk}

\begin{document}
	
\maketitle

\begin{abstract}
  We consider a drift-diffusion model, with an unknown function
  depending on the spatial variable and an additional structural
  variable, the amount of ingested lipid. The diffusion
  coefficient depends on this additional variable. The drift acts
  on this additional variable, with a power-law coefficient of
  the additional variable and a localization function in
  space. It models the dynamics of a population of macrophage
  cells. Lipids are located in a given region of space; when
  cells pass through this region, they internalize some
  lipids. This leads to a problem whose mathematical novelty is
  the dependence of the diffusion coefficient on the past
  trajectory. We discuss global existence and blow-up of the
  solution.
\end{abstract}

	
\section{Introduction}

\subsection{The problem at hand}

Atherosclerosis is a major cause of death in industrialized
societies since it is the primary cause of heart attack (acute
myocardial infarction) and stroke (cerebrovascular accident). It
is now accepted that atherosclerosis is a chronic inflammatory
disease which starts within the intima, the innermost layer of an
artery. It is driven by the accumulation of macrophage cells
within the intima and promoted by modified low density
lipoprotein (LDL) particles~\cite{Libby,Tabas}. Inflammation
occurs at sites within the arterial wall where modified
low-density lipoproteins (LDL) accumulate after penetrating the
wall from the bloodstream. The immune response attracts
circulating monocytes to these sites.
	
Mathematical modeling of atherosclerosis has recently gained
interest because of the variety of behavior of macrophages
depending on the amount of lipid ingested~\cite{Moore}. Many
mathematical models divide the macrophage population into
``macrophages'' and ``foam cells''. This is true in models based
on ordinary differential equations~\cite{Cohen}, spatially
resolved partial differential equations
models~\cite{Chalmers,Hao_1,Hao_2,Calvez_1,Calvez_2} and other
computational models~\cite{Parton}. However, it is now clear that
there is a continuous distribution of lipid loads in macrophage
populations, ranging from monocytes with low lipid loads to
macrophages that can be labeled as foam cells in
atherosclerosis. The variation in lipid load within a macrophage
population suggests that it may be instructive to develop
structured models in which macrophages are characterized by their
intracellular lipid load, similar to adipocytes~\cite{Soula}, in
the spirit of~\cite{Muller,Byrne}.
	
We consider here the case where the arterial wall is $\R^d$ with
$d\ge 1$ and lipids are localized in a prescribed region defined
by $\operatorname*{Supp}\psi$ where $\psi$ is a given
non-negative function of space. We describe the dynamics of
macrophages in the arterial wall by a statistical (kinetic)
description of the evolution of their density with respect to
the spatial and lipid variables: 
\begin{equation}
  \begin{cases}
    \partial_{t}u=a\Delta_{x}u+ \partial_{a}\left\{
      \left[ \psi (x)f(a)- \lambda a(A^*-a)\right] u\right\},  & x
    \in \R^d, a\ge 0, t>0, \\[2mm]
    u_{|t=0}=u_{0}, & x \in \R^d, a\ge 0,
  \end{cases}
  \label{Pacman_general}
\end{equation}
where $A^*>0$, $\lambda \ge 0$ and $f$ is a given function.

In~\eqref{Pacman_general}, $u(t,x,a)$ is the density of cells
(macrophages) located at time $t$ in $x\in \R^d$ and having with
an amount $a$ of ingested lipids. The amount of lipids in the
cell affects its spatial dynamics by changing its size for
example. This effect is modeled by a diffusion coefficient
depending on the variable $a$. The term $\lambda a(A^*-a)$ is a
logistic (dumping) term that describes the recovery effect,
i.e. the tendency of a cell to return to its normal state as it
moves away from the hot spot.
	 
We address the question of whether or not this
advection-diffusion equation can lead to concentration and
eventually blow-up in finite time, if diffusion is not strong
enough to prevent cells from being trapped in lipid dense
regions. More precisely, our purposes here are to investigate the
influence of the nonlinearity $f$, the parameter $\lambda$ and
the size of the support of $\psi$ on the local or global
well-posedness of problem~\eqref{Pacman_general}. To this end we
will consider functions $f$ which typically grow like a power,
namely
\begin{equation*}
  f(a)=\pm a ^\gamma  \textrm{ with } \gamma >0,  
\end{equation*}
and three different cases for the function $\psi$:
\begin{enumerate}
\item $\psi=\mathds{1}_\R$,
\item $\operatorname*{Supp}\psi\left( x\right) \subset B(0,1)$
  with $\int_{\mathbb{R}^d}\psi\left( x\right) \dx=1$,
\item $\psi = \delta_{x=0}$ is the Dirac mass in dimension $d=1$,
  which was treated in \cite{M} when $\lambda =0$.
\end{enumerate}	

\subsection{The origin of the model}

Let us briefly detail the origin of equation
\eqref{Pacman_general} in the case where $d=1$, see \cite{B,M,EM}
for more details. Consider a cell, described as a Brownian
particle, with position $X_t$, whose diffusion coefficient,
$A_t$, is modified at each passage in the lipid-rich zone and
which tends to recover a normal diffusion coefficient, see
Figure~\ref{macrophage}. The couple $(X_t,A_t)$ satisfies the
system of stochastic differential equations
\begin{equation}
  \label{new}
  \begin{cases}
    \D X_t &= \sqrt{2A_t} \D W_t, \\[2mm]
    \D A_t &= - \psi(X_t) f(A_t) \dt + \ \lambda A_t \left(
      A^* - A_t\right) \dt,
  \end{cases}
\end{equation}
with initial condition $X_0\in \R$, $A_0 >0$, where
$(W_t, t\geq 0)$ is a one-dimensional Brownian motion.

\begin{figure}
  \begin{center}
    \includegraphics[scale=0.5]{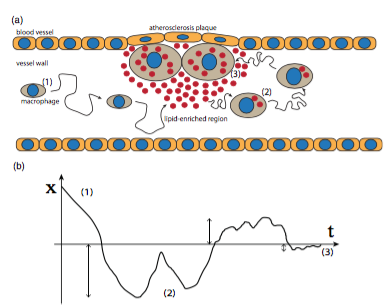}
  \end{center}
  \caption{a) Sketch of the different stages of atherosclerosis
    plaque formation: (1) rapid diffusion of a “free” macrophage
    cell; (2) upon entering a localized lipid-enriched region,
    the macrophage accumulates lipids, and thereby grows and
    becomes less mobile; and (3) after many crossings of the
    lipid-enriched region, the macrophage eventually gets
    trapped, resulting in the formation of an atherosclerotic
    plaque. (b) Sketch of a one-dimensional particle trajectory
    of the model of locally decelerated random walk.}
	\label{macrophage}
\end{figure}
	
Intuitively, $ \int_0^t \psi(X_s) \ds$ represents the time spent
by $X$ in the lipid-rich region during $[0,t]$. Hence, when
$f \ge 0$, the time spent in this region by the macrophage cell
tends to decrease its diffusion coefficient $(A_t, t\geq 0)$, or
equivalently to decelerate the cell. Such a decrease in local
mobility is due to the internalization and accumulation of lipid
particles by macrophages. It can be related to several phenomena:
increase in volume or change in cell fate depending on the amount
of low-density lipoproteins (LDL) ingested. 
	
The parameter $\gamma$ describes the intensity of the
internalization: the higher $\gamma$ is, the lower the
internalization is. This parameter is thus related to both the
amount of lipids located in the arterial wall hotspot and the
capacity of macrophages to internalize lipids. The case
$\gamma =0$ corresponds to a very large amount of lipids and a
very large capacity of macrophages to internalize lipids. In
summary, the parameter $\gamma $ is related to the inflammatory
state. In contrast, the term
$\lambda A_t \left( A^* - A_t\right) \dt$ describes the loss of
lipids by the macrophage and its tendency to recover its natural
state. The parameters $\lambda$ and $A^*$ model the natural
immune defenses or medical treatment.
	
System~\eqref{new} models the effects of lipid accumulation on
macrophage dynamics. The novelty of~\eqref{new}, compared with
the related model in \cite{B,M}, is to include different
geometries for lipid source locations (from the Dirac case to the
whole real line) and to include the lipid unloading through the
logistic dumping term. Non-trivial behaviors of the solution of
\eqref{new} are expected: a dynamic transition to an absorbing
state can occur for a sufficiently strong deceleration, and the
particle can be trapped in finite time; whereas if the
deceleration process is sufficiently weak, the particle is never
trapped.

\subsection{The main results}

Since the couple $(X_t,A_t)$ is Markovian, in order to make the
previous intuition rigorous, we study the partial differential
equation~\eqref{Pacman_general} and the general questions we are
concerned with are the following. By studying the regularity of
the solution to~\eqref{Pacman_general}, do we recover the results
expected by the probabilistic description, in particular by the
simplified differential equations $A'(t) = - f(A(t))$? In
particular, if $f (a) = a^\gamma$ with $\gamma \ge 1$, can we
prove global existence as expected from the previous differential
equation? In the case $f (a) = a^\gamma$ with $\gamma \in (0,1)$,
can we prove that the solution to~\eqref{Pacman_general} becomes
unbounded in finite time in any $L^p$ space? In order to ease the
presentation, we will discuss the particular case of space
dimension 1. However, our results remain true in dimension 2. In
Section~\ref{sec:2d} we explain the changes in order to recover
the blow-up for $\gamma \in [0,1)$ in the the two-dimensional
case.

Let us build some intuition on~\eqref{Pacman_general} in the case
where $f(a)=a^\gamma$, $\psi=\mathds{1}_\R$, $\lambda =0$ and
$d=1$. Multiplying~\eqref{Pacman_general} with $a^\gamma$ and
denoting $m(t,x,a)=a^\gamma u(t,x,a)$, we obtain:
\begin{equation}
  \label{edp_droite_2}
  \partial_t m(t,x,a)=a  \partial_{xx}m(t,x,a) +a^\gamma
  \partial_ a m(t,x,a) \, .
\end{equation}  

Let us assume that $\partial_x m(t,x,a)$ vanishes at
$x \sim \pm \infty$ and let us integrate
equation~\eqref{edp_droite_2} on $\R$:
\begin{equation}
  \label{edp_droite_3}
  \partial_t \int_{\R} m (t,x,a)\D x-a^\gamma \partial_ a
  \int_{\R} m(t,x,a) \D x =0\, .
\end{equation}  
We then write the characteristics $a(t)$ of the previous system,
solution to the equation:
\begin{equation}
  \label{eq:caracteristiques}
  \begin{cases}
    a'(t)&=-a^\gamma (t)\\
    a(0)&=a_0 .
  \end{cases}
\end{equation}
The equation~\eqref{eq:caracteristiques} shows how the diffusion
coefficient evolves (note that it changes in the same way for all
individuals) namely we identify two different regimes:
\begin{itemize}
\item if $\gamma \in [0,1)$ the diffusion coefficient vanishes in
  finite time,
\item if $\gamma \ge 1$ the diffusion coefficient is positive for
  all time provided the initial datum $a_0 >0$, and vanishes
  asymptotically in time as $t \to +\infty$.
\end{itemize}

We are now in the position of stating our main results. We begin
with the case where the acceleration-deceleration zone is the
whole real axis:
\begin{equation}
  \label{pacman:droite}
  \begin{cases}
    \partial_{t}u=a\partial_{xx}^{2}u+ \partial_{a}\left(
      f(a)u\right), & x \in \R, \ a\ge 0, \ t>0,\\[2mm]
    u_{|t=0}=u_{0}, & x \in \R, \ a\ge 0.
  \end{cases}
\end{equation}

We consider the three following cases:
\begin{itemize}
  \itemsep0em 
\item[(a)] acceleration: $f(a) = -a^\gamma$ with $\gamma \ge 0$,
\item[(b)] subcritical deceleration: $f(a) = a^\gamma$ with
  $\gamma \geq 1$,
\item[(c)] supercritical deceleration: $f(a) = a^\gamma$ with
  $\gamma \in [0,1)$.
\end{itemize}

In particular, if $f(a)=a^\gamma$, by studying the regularity of
the solution in a $L^p$ framework, we obtain the following
dichotomy: global existence if $\gamma \ge 1$, while, in the case
$\gamma<1$, the solution becomes unbounded in finite time
(so-called blow-up).

This is a linear equation on $u=u(t,x,a)$ defined on $t \ge 0$,
$x \in \R$, $a \ge 0$. 

The first main result reads :
\begin{theo}
  \label{th:edp}
  Assume that the initial datum $u_0$ belongs to $L^p\cap L^1$,
  $p\geq1$. Assume in addition that
  $\mbox{supp}(u_0) \subset \R \times [0,A_0]$ for some $A_0>0$,
  then
  \begin{itemize}
  \item[(a)] in the acceleration case $f(a) = -a^\gamma$ with
    $\gamma \ge 0$, there exists a unique global weak solution
    to~\eqref{pacman:droite} that satisfies
    \begin{equation*}
      \forall \, t \ge 0, \quad \int_{\R} \int_{\R_+} |u(t,x,a)|^p \dd x
      \dd a \le \int_{\R} \int_{\R_+} |u(0,x,a)|^p \dd x \dd a,
    \end{equation*}
  \item[(b)] in the subcritical deceleration case
    $f(a) = a^\gamma$ with $\gamma \ge 1$, there exists a unique
    global weak solution to~\eqref{pacman:droite} that satisfies
    \begin{equation*}
      \forall \, t \ge 0, \quad \int_{\R} \int_{\R_+} |u(t,x,a)|^p \dd x
      \dd a \le e^{\gamma (1-p) A_0^{\gamma-1} t} \int_{\R}
      \int_{\R_+} |u(0,x,a)|^p \dd x \dd a,
    \end{equation*}
  \item[(c)] in the supercritical deceleration case
    $f(a) = a^\gamma$ with $\gamma \in [0,1)$, any non-zero weak
    solution to~\eqref{pacman:droite} blows-up in finite time
    with the time of blow-up
    \begin{equation*}
      T_b := \frac{\int_\R \int_{\R_+} a^{1-\gamma} u(0,x,a)
        \dxa}{(1-\gamma) \left( \int_\R \int_{\R_+} u(0,x,a)
          \dxa \right)} \in (0,+\infty)
    \end{equation*}
    for which all the mass is concentrated at $a=0$. 
  \end{itemize}
\end{theo}

Our second result concerns the case where the
acceleration-deceleration zone is localized:
\begin{equation}
  \begin{cases}
    \partial_{t}u=a\partial_{xx}^{2}u+\psi\left( x\right)
    \partial_{a}\left(
      f(a)u\right),  & x \in \R, \ a\ge 0, \ t>0,\\[2mm]
    u_{|t=0}=u_{0}, & x \in \R, \ a\ge 0,
  \end{cases}
  \label{Pacman_prototip}
\end{equation}
with a localization function that satisfies
\begin{equation*}
  \operatorname*{Supp}\psi\left( x\right) \subset\left[
    -1,1\right] \quad \text{ and } \quad
  \int_{\mathbb{R}}\psi\left( x\right) \dx=1.
\end{equation*}
We will consider the three regimes described above. 
\begin{theo}
  \label{Blow-up_viriel}
  Given $u_0 \in L^1 \cap L^p$, $p\geq1$, with $u_0 \ge 0$ and
  $\supp u_0 \subset \R \times [0,A_0]$, $A_0>0$, we have:
  \begin{itemize}
  \item[(a)] In the acceleration case, $f(a) = -a^\gamma$ with
    $\gamma >0$, there exists a unique weak solution
    to~\eqref{pacman:interval} that satisfies
     \begin{equation*}
      \forall \, t \ge 0, \quad \int_{\R} \int_{\R_+} |u(t,x,a)|^p \dd x
      \dd a \le \int_{\R} \int_{\R_+} |u(0,x,a)|^p \dd x \dd a,.
    \end{equation*}
  \item[(b)] In the subcritical deceleration case
    $f(a) = a^\gamma$ with $\gamma \ge 1$, there exists a unique
    weak solution to~\eqref{pacman:interval} that satisfies
    \begin{equation*}
      \forall \, t \ge 0, \quad \int_{\R} \int_{\R_+} |u(t,x,a)|^p \dd x
      \dd a \le e^{\frac{\gamma (p-1)
    A_0^{\gamma-1}}{2\delta} t} \int_{\R} \int_{\R_+} |u(0,x,a)|^p \dd x \dd a.
    \end{equation*}
    Moreover, when $\gamma>3/2$, the following bound (uniform in
    $\delta$) holds
    \begin{equation*}
      \forall \, t \ge 0, \quad \int_{\R} \int_{\R_+} |u(t,x,a)|^p \dd x
      \dd a \le e^{\gamma^2(p-1) t} \int_{\R} \int_{\R_+} |u(0,x,a)|^p \dd x \dd a.
    \end{equation*}
  
  \item[(c)] in the supercritical deceleration case $f(a) = a^\gamma$ with $\gamma \in [0,1)$, provided that $a_0$ is small, there exists arbitrarily small non-negative initial data
  \begin{equation*}
    0 \le u_{0}\in L^{1}\cap L^p ( \R \times \R_+ ) ,\quad
    \operatorname*{Supp}u_{0}\subset\mathbb{R}\times\lbrack
    a_{0},+\infty)
  \end{equation*}
  so that the corresponding unique weak solution
  $0 \le u \in L^1 \cap L^p$ to~\eqref{Pacman_prototip} blows-up
  in finite time $T^{\ast} \in (0,+\infty)$, in the sense
  \begin{equation}
    \label{accumulation}
    \limsup_{t\rightarrow T^{\ast}}\left\Vert u\left(  t\right)
    \right\Vert_{L^p(\R\times \R_+)}=+\infty.
  \end{equation}
  Moreover this time of explosion can be as small as wanted   provided that $a_0$ is small enough.
  \end{itemize}
\end{theo}

\subsection{Structure of the paper}

The remainder of this paper is organised as follows. In
Section~\ref{sec:line} we study the case where the diffusion
coefficient changes on the whole real line and $\lambda = 0$. If
$f(a)= a ^\gamma$, we show that $\gamma=1$ plays a critical role
in the problem. In Section~\ref{sec:interval} we consider the
case where the diffusion coefficient changes on an interval and
$\lambda = 0$. We prove existence and uniqueness if
$\gamma \ge 1$ and blow-up of the solution if $\gamma \in
(0,1)$. In Section~\ref{sec:autre_models}, we consider variants
of~\eqref{Pacman_prototip} in three directions: (1) with a
nonlinear diffusion operator, (2) with logistic memory (dumping),
and finally in two dimensions in Subsection~\ref{sec:2d}.

\section{Acceleration-deceleration on $\R$}
\label{sec:line}

In this section we prove Theorem \ref{pacman:droite} regarding
the case where the acceleration-deceleration zone is the whole
real axis:
\begin{equation}
  \label{pacman:droitei}
  \begin{cases}
    \partial_{t}u=a\partial_{xx}^{2}u+ \partial_{a}\left(
      f(a)u\right), & x \in \R, \ a\ge 0, \ t>0,\\[2mm]
    u_{|t=0}=u_{0}, & x \in \R, \ a\ge 0.
  \end{cases}
\end{equation}

We start by defining weak solutions to~\eqref{pacman:droite}.
\begin{defi}
  \label{def:weak}
  Let $A_0,T>0$ and
  $u_{0}\in L^{1}\left( \mathbb{R}^{2}\right) \cap L^{p}\left(
    \mathbb{R}^{2}\right) $ such that
  \begin{equation*}
    \operatorname*{Supp}u_{0}\left( x,a\right)
    \subset\mathbb{R}\times \lbrack0,A_{0}]\text{ and
    }u_{0}\geq0.
  \end{equation*}
  A function
  $u\in C(\mathbb{[}0,T\mathbb{)};L^{1}(\mathbb{R}^{2}))\cap
  C(\mathbb{[}0,T\mathbb{)};L_{\mathrm{weak}}^{p}(\mathbb{R}^{2}))$ such
  that
  \begin{align*}
    & u \in
      L^{\infty}(\mathbb{(}0,T\mathbb{)};L^{p}(\mathbb{R}^{2})),a\partial
      _{x}u^{\frac{p}{2}}\in L^{2}\left(  \left(  0,T\right)
      \times\mathbb{R}^{2}\right) ,\\
    & \operatorname*{Supp}u\left( t,x,a\right)
      \subset\mathbb{R}\times \lbrack0,A_{0}]\text{ and }u\left(
      t,x,a\right) \geq0\text{,}
  \end{align*}
  verifying~\eqref{pacman:droite} in the sense of distributions
  is called a weak solution to~\eqref{pacman:droite} on
  $\left( 0,T\right) $. A function
  $u\in C(\mathbb{R}_{+};L^{1}(\mathbb{R}^{2}))\cap
  C(\mathbb{R}_{+}%
  ;L_{\mathrm{weak}}^{p}(\mathbb{R}^{2}))$ is called a global weak
  solution to~\eqref{pacman:droite} if for any $T>0$ its
  restriction to $\left( 0,T\right) $ is a weak solution
  to~\eqref{pacman:droite} on $\left( 0,T\right) $.
\end{defi}

Weak solutions in the sense of Definition~\ref{def:weak} are
mass-preserving:
\begin{equation*}
  M =\int_\R \int_{\R_+} u_0(x,a)\dxa = \int_\R \int_{\R_+}
  u(t,x,a)\dxa.
\end{equation*}

Let us first prove that non-negativity is preserved.
\begin{lemm}
  Assume that $u$ is a weak solution to~\eqref{pacman:droite}
  with initial datum $u_0$ non-negative almost everywhere.  Then
  $|u(t,\cdot)|=u(t, \cdot)$ almost everywhere for later times.
\end{lemm}

\begin{proof}
  If $u$ is solution in $L^1$ then $|u|$ is subsolution in $L^1$
  since $\mbox{sgn}(u) \partial_{xx}^2 u \le \partial_{xx}^2 |u|$
  and
  $\mbox{sgn}(u) \partial_a ( f(a) \, u ) = \partial_a ( f(a) \,
  |u| ) $. Hence $|u|-u$ is a subsolution, and we calculate
  \begin{align*}
    \ddt \int_\R \int_{\R_+} \left( |u| - u \right)
    \dxa \le 0,
  \end{align*}
  which proves the claimed conclusion.
\end{proof}

We then use the orientation of the drift to obtain estimates on
the $(x,a)$-support of the solution.
\begin{lemm}
  \label{lem:support_compact}
  Assume $u$ is a weak solution to~\eqref{pacman:droite} with
  $u_0 \ge 0$.  Assume in addition that $f(a) \ge 0$
  (deceleration) and $\mbox{supp}(u_0) \subset \R \times [0,A_0]$
  for some $A_0>0$, then
  $\mbox{supp}(u_t) \subset \R \times [0,A_0]$ up to the existence
  time. Similarly, in the case $f(a) \le 0$ (acceleration) and
  $\mbox{supp}(u_0) \subset \R \times [A_0,+\infty)$ for some
  $A_0>0$, then $\mbox{supp}(u_t) \subset \R \times [A_0,+\infty)$
  for some $A_0>0$ up to the existence time.
\end{lemm}

\begin{proof}
  Note first that $u_0 \ge 0$ implies $u \ge 0$ by standard
  arguments. Consider then any non-negative non-decreasing function
  $\varphi=\varphi(a)$ smooth on $\R_+$ with support included in
  $(A_0,+\infty)$. Then
  \begin{align*}
    \ddt \int_{\R} \int_{\R_+} u(t,x,a) \varphi(a) \dd
    x \dd a = \int_{\R} \int_{\R_+} f(a) u(t,x,a)
    \varphi'(a) \dxa \le 0,
  \end{align*}
  which proves that $u_t \varphi =0$ for later times. This proves
  that $u_t=0$ on $\R \times [A_0,+\infty)$ for all times.
\end{proof}

\begin{proof}[Proof of Theorem~\ref{th:edp}]
  This is a standard drift-diffusion equation and we omit the
  details regarding the existence of solutions, see for
  instance~\cite{kawashima,AG}.
	 
  In the cases (a) and (b), we prove the propagation of $L^p$
  bounds for all times, which is the crucial a priori
  estimate. To prove that solutions blow-up in finite time in the
  supercritical case (c), we show that for an appropriate value
  of $m>0$, the moment $\int_{\R_+} a^m u(t,x,a) \dd a $ becomes
  infinite in finite time.
  	
  \noindent
  \textbf{Case (a).} Given $p \in [1,+\infty)$, we calculate
  \begin{align*}
    \ddt \int_\R \int_{\R_+} |u|^p \dd x \dd a \le
    & - p(p-1) \int_{\R} \int_{\R_+} a \left| \partial_x u
      \right|^2 |u|^{p-2} \dd x \dd a \\
    & - \gamma(p-1) \int_\R \int_{\R_+} a^{\gamma-1} |u|^p
      \dxa  \le 0,
  \end{align*}
  which proves that $L^p$ norms remains finite for all times
  provided they are finite initially, and there is no blow-up in
  finite time. By applying the same argument to the modulus of
  the difference of two solutions one proves similarly uniqueness
  in $L^p$.
	
  \indent \textbf{Case (b).} Given $p \in [1,+\infty)$, we
  calculate
  \begin{equation*}
    \ddt \iint_{\R\times \R_+} |u|^p \dd x \dd a
     \le - p(p-1) \iint_{\R\times \R_+} a \left| \partial_x u
      \right|^p  \dd x \dd a 
     + \gamma (p-1)  \iint_{ \R\times \R_+} a^{\gamma-1} |u|^p
      \dxa.
  \end{equation*}
  Since $\gamma \ge 1$, the weight $a^{\gamma-1}$ is bounded on
  $[0,A_0]$, using Lemma~\ref{lem:support_compact}, and we deduce
  \begin{equation*}
    \ddt \iint_{ \R\times \R_+} |u|^p \dxa \le \gamma (p-1)
    A_0^{\gamma-1} \iint_{ \R\times \R_+} |u|^p \dxa,
  \end{equation*}
  which implies the propagation of $L^p$ norms by Gronwall lemma
  and the bound in the statement.
	
  \indent \textbf{Case (c).} The conclusion follows from
  \begin{align*}
    \ddt \int_\R \int_{\R_+} a^{1-\gamma} u(t,x,a) \dxa
    & = \int_\R \int_{\R_+} a^{1-\gamma} \partial_a \left(
      a^\gamma u \right) \dxa \\
    &   = - (1-\gamma) \int_\R \int_{\R_+} u(t,x,a) \dxa \\
    & = - (1-\gamma) \int_\R \int_{\R_+} u(0,x,a) \dxa,
  \end{align*}
  where we have used the conservation of mass
  $\iint u(t,x,a) \dxa = \iint u_0(x,a) \dxa$.
\end{proof}	

\section{Acceleration-deceleration on an interval}
\label{sec:interval}

In this section we study, for $f$ locally Lipschitz on
$(0,\infty)$,
\begin{equation}
  \begin{cases}
    \partial_{t}u=a\partial_{xx}^{2}u+\psi\left( x\right)
    \partial_{a}\left(
      f(a)u\right),  & x \in \R, \ a\ge 0, \ t>0,\\[2mm]
    u_{|t=0}=u_{0}, & x \in \R, \ a\ge 0,
  \end{cases}
  \label{Pacman_prototipi}
\end{equation}
with a localization function that satisfies
\begin{equation*}
  \operatorname*{Supp}\psi\left( x\right) \subset\left[
    -1,1\right] \quad \text{ and } \quad
  \int_{\mathbb{R}}\psi\left( x\right) \dx=1.
\end{equation*}

\subsection{Comparison with the acceleration-deceleration at a
  point}

Formally, solutions to~\eqref{Pacman_prototip}, for a family of
functions $\psi_n$ that approximates the Dirac mass at zero,
converge to the model introduced in~\cite{M} which reads
\begin{equation}
  \label{eq:dirac}
  \begin{cases}
    \partial_{t}u=a\partial_{xx}^{2}u+\delta_{x=0}\partial_{a}\left(
      f(a)u\right) ,
    & x \in \R, \ a\ge 0, \ t>0,\\[2mm]
    u_{|t=0}=u_{0}, & x \in \R, \ a\ge 0,
  \end{cases}
\end{equation}
However, the method of~\cite{M} does not apply. Consider for
instance the case $f(a)= a^\gamma$ with $\gamma \in (0,1)$. In
\cite{M}, it was shown that any solution $u$ to~\eqref{eq:dirac}
satisfies
\begin{equation}
  \int_{0}^{\infty}a^{\frac{1}{2}}u\left(  t,0,a\right)  \da<+\infty
  \label{quantitty}
\end{equation}
for all times, and that for any $m\in (0,1/2)$ there is
$T^{\ast}_m \in (0,+\infty)$ such that
\begin{equation*}
  \lim_{t\rightarrow T^{\ast}_m}\int_{0}^{\infty}a^{m}u\left(
    t,0,a\right) \da=+\infty.
\end{equation*}

Looking now at the equation~\eqref{Pacman_prototip}, the
equivalent of the quantity \eqref{quantitty} is
\begin{equation*}
  \int_{0}^{\infty}\int_{-\infty}^{\infty}\psi\left( x\right)
  a^{m}u\left( t,x,a\right) \dxa,
\end{equation*}
and, when $m=1/2$, we are able to find a bound similar
to~\eqref{quantitty}:
\begin{lemm}
  Assume that $u$ is a weak solution to~\eqref{Pacman_prototip}
  with $f(a)=a^\gamma$ and $\gamma \in (0,1)$ (in the sense of
  Definition~\ref{def:weak}), then
  \begin{equation*}
    \int_\R \int_{\R_+} \psi\left( x\right)
    a^{\frac12}u\left( t,x,a\right) \dxa < \infty.
  \end{equation*}
\end{lemm}

\begin{proof}
  Using a Duhamel formulation along the heat flow, we
  have
  \begin{equation}
    \label{form_mild}
    u\left( t,\cdot,a\right)
    =\frac{\sqrt{\pi}}{\sqrt{ta}}e^{\frac{-\left\vert
          \cdot\right\vert ^{2}}{4ta}}\ast_{x}u_{0}\left(
      \cdot,a\right) +\int_{0}%
    ^{t}\frac{\sqrt{\pi}}{\sqrt{\tau a}}e^{\frac{-\left\vert
          \cdot\right\vert ^{2}}{4\tau
        a}}\ast_{x}\partial_{a}\left( \psi
      a^{\gamma}u)\right),
  \end{equation}
  which implies, after integration against $a^m \psi$,
  \begin{align*}
    &\int_\R \int_{\R_+} \psi\left(  x\right)
    a^{m}u\left( t,x,a\right)  \dx
     =\int_\R \int_{\R_+} \psi\left(
      x\right)  a^{m} \left(
      \frac{\sqrt{\pi}}{\sqrt{ta}}e^{\frac{-\left\vert \cdot
      \right\vert ^{2}}{4ta}} \ast_{x}u_{0}\left(  \cdot,a\right)
      \right) \left(  x\right) \dxa \\
    & \qquad +\int_{0}^{t} \int_\R \int_{\R_+} \psi\left(
      x\right) a^{m} \left( \frac{\sqrt{\pi}}{\sqrt{\tau
      a}}e^{\frac{-\left\vert \cdot\right\vert^{2}}{4\tau
      a}}\ast_{x}\partial_{a}\left(  \psi a^{\gamma}u)\right)
      \right)(x) \dtau \dxa.
  \end{align*}
	
  Observe that the first term in the previous right hand side is
  bounded for $u_0 \in L_{x}^{\infty} L_{a}^{1}(a^m)$:
  \begin{equation*}
    \int_\R \int_{\R_+} \psi\left( x\right)
    a^{m} \left(\frac{\sqrt{\pi}}{\sqrt{ta}}e^{\frac{-\left\vert
            \cdot\right\vert ^{2}}{4ta}}
      \ast_{x}u_{0}\left( \cdot,a\right)\right) \left(
      x\right) \dxa\leq\int _{0}^{+\infty}a^{m}\left\Vert
      u_{0}\left( \cdot,a\right) \right\Vert
    _{L_{x}^{\infty}}\da<\infty.
  \end{equation*}
  
  Next, we compute
  \begin{align*}
    & \int_{0}^{t} \int_\R \int_{\R_+} \psi\left(
      x\right) a^{m} \left( \frac{\sqrt{\pi}}{\sqrt{\tau
      a}}e^{\frac{-\left\vert \cdot\right\vert^{2}}{4\tau
      a}}\ast_{x}\partial_{a}\left(  \psi a^{\gamma}u\right)
      \right)(x) \dtau \dxa  \\
    &  =\sqrt{\pi}\int_{0}^{t} \int_\R \int_{\R_+} 
      \int_\R \psi\left(  x\right)  a^{m-\frac{1}{2}}\frac
      {e^{\frac{-\left\vert x-y\right\vert ^{2}}{4\tau a}}}{\sqrt{\tau}}\partial
      _{a}\left[  \psi\left(  y\right)  a^{\gamma}u\left(
      \tau,y,a\right)  \right] \dtau \dxa \dy \\
    &  =\left(  \frac{1}{2}-m\right)
      \sqrt{\pi}\int_{0}^{t} \int_\R \int_{\R_+} \int_\R a^{\gamma+m-\frac{3}{2}}
      \frac{e^{\frac{-\left\vert x-y\right\vert ^{2}}{4\tau
      a}}}{\sqrt{\tau}} \psi\left(  x\right)  \psi\left(
      y\right)  u\left(  \tau,y,a\right) \dtau \dxa \dy \\
    &  -\sqrt{\pi}\int_{0}^{t} \int_\R \int_{\R_+} \int_\R
      a^{\gamma+m-\frac{1}{2}}\frac{\partial}{\partial
      a}\left(\frac{e^{\frac{-\left\vert x-y\right\vert
      ^{2}}{4\tau a}}}{\sqrt{\tau}}\right) \psi\left(  x \right)
      \psi\left(  y \right)  u\left(  \tau,y,a\right) \dtau \dxa \dy.
  \end{align*}
  The last term in the previous equality can be written
  \begin{equation*}
    - \sqrt{\pi} \int_{0}^{t}\int_\R \int_{\R_+} \int_\R
    \psi\left(  x\right)  a^{\gamma+m-\frac{1}{2}}\frac{e^{\frac
      {-\left\vert x-y\right\vert ^{2}}{4\tau
      a}}}{\sqrt{4\tau}}\frac{\left\vert x-y\right\vert
      ^{2}}{4\tau a^{2}}\psi\left(  y\right)  u\left(
      \tau,y,a\right) \dtau \dxa \dy \leq 0.
  \end{equation*}
  Thus, for $m=\frac{1}{2}$, we have
  \begin{equation*}
    \int_\R \int_{\R_+}  \psi\left( x\right)
    a^{\frac{1}{2}}u\left( t,x,a\right)
    \dxa\leq\int_{0}^{+\infty}a^{\frac{1}{2}}\left\Vert
      u_{0}\left( \cdot,a\right) \right\Vert
    _{L_{x}^{\infty}}\da<\infty
  \end{equation*}
  which concludes the proof.
\end{proof}

However, unlike for equation~\eqref{eq:dirac}, one cannot hope
that the quantity
\begin{equation*}
  \int_\R \int_{\R_+} \psi\left( x\right) a^{m}u\left( t,x,a\right) \dxa
\end{equation*}
becomes infinite for $m\in\left( 0,\frac{1}{2}\right) $ since, if
$u$ is compactly supported in $a$ initially, say $u_0(x,a)=0$ for
$a>A_0$ with $A_0>0$, it is controlled by the mass:
\begin{equation*}
  \int_\R \int_{\R_+} \psi\left( x\right)
  a^{m}u\left( t,x,a\right) \dxa\leq\left\Vert \psi\right\Vert
  _{L^{\infty}}A_0^{m}\int_\R \int_{\R_+} u\left( t,x,a\right) \dxa.
\end{equation*}
Of course, this estimate degenerates for $\psi_n$ converging to a
Dirac mass at $x=0$, since $\| \psi_n \|_\infty \to \infty$. 

\subsection{The cases without finite time singularity}



Denote
$\psi_\delta(x) := (2\delta)^{-1} {\mathds 1}_{x \in
  [-\delta,\delta]}$ and consider
\begin{equation}
  \label{pacman:interval}
  \begin{cases}
    \partial_t u = a \partial^2_{xx} u + \psi_\delta  \partial_a
    \left( f(a) \, u \right),
    & x \in \R, \ a\ge 0, \ t>0,\\[2mm]
    u_{|t=0}=u_{0}, & x \in \R, \ a>0.
  \end{cases}
\end{equation}

\begin{proof}[Proof of Theorem~\ref{Blow-up_viriel} for the cases
  (a) and (b)]
  We prove the a priori estimate, then briefly explain how to
  construct solutions.
  
  \textbf{Case (a).} The estimate follows from
  \begin{align*}
    \ddt \int_\R \int_{\R_+} |u|^p \dxa \le
     & - p(p-1) \iint_{\R \times \R_+} a \left| \partial_x u
      \right|^2 |u|^{p-2} \dxa \\
     & - \frac{\gamma (p-1)}{2\delta} \int_{[-\delta,\delta]} \int_{\R_+}
    a^{\gamma-1} |u|^p \dxa \le  0.
  \end{align*}

  \textbf{Case (b).} We calculate
  \begin{align*}
    \ddt \int_\R \int_{\R_+} |u|^p \dxa
    \le & - p(p-1) \int_\R \int_{\R_+} a \left| \partial_x u
          \right|^2 |u|^{p-2} \dxa \\
    & + \frac{\gamma (p-1)}{2\delta} \int_{x \in [-\delta,\delta]} \int_{a
    \in \R_+} a^{\gamma-1} |u|^p \dxa.
  \end{align*}
  Since $\gamma \ge 1$, the weight $a^{\gamma-1}$ is bounded on
  $[0,A_0]$ by using Lemma~\ref{lem:support_compact}, and we
  deduce
  \begin{align*}
    \ddt \int_\R \int_{\R_+} |u|^p \dxa \le \frac{\gamma (p-1)
    A_0^{\gamma-1}}{2\delta} \int_\R \int_{\R_+} |u|^p \dxa,
  \end{align*}
  which proves the estimate by the Gronwall lemma. When $p>1$, we
  further calculate
  \begin{align*}
    \frac{\D}{\dt} \iint_{ \R\times \R_+} |u|^p \dxa =
    & - \frac{4(p-1)}{p} \iint_{
      \R\times \R_+} a \left|\partial_x \left( u^{\frac{p}2}
      \right) \right|^2\dxa \\
    & + \frac{\gamma(p-1)-m}{2 \delta} \int_{x \in [\delta,\delta]} \int_{a \in \R_+}
    a^{\gamma-1} |u|^p \dxa
  \end{align*}
  and we control, with $I_\var := [x-\var,x+\var]$ and $v :=
  u^{p/2}$, 
  \begin{align*}
    |v(t,x,a)|
    & \le \left| v(t,x,a) - \frac1{|I_\var|} \int_{y \in I_\var}
      v(t,y,a) \right| 
      + \left| \frac1{|I_\var|} \int_{y \in I_\var}
      v(t,y,a) \D  y \right| \\
    & \le \left|  \frac1{|I_\var|} \int_{y \in I_\var}
      \Big( v(t,x,a) - v(t,y,a) \Big) \D  y \right| 
      + \frac{1}{\sqrt{2\var}} \| v(t,\cdot,a) \|_{L^2_x(\R)} \\
    & \le \left|  \frac1{2 \var} \int_{-\var} ^\var \int_x ^y 
      \partial_x v(z,a) \D  y \D  z \right|
      + \frac{1}{\sqrt{2\var}} \| v(t,\cdot,a) \|_{L^2_x(\R)} \\
    &\le \sqrt{2 \var} \left\| \partial_x
      v(\cdot,a) \right\|_{L^2_x(\R)} 
      + \frac{1}{\sqrt{2\var}} \| v(\cdot,a) \|_{L^2_x(\R)}
  \end{align*}
  and we deduce (note that the previous bound is pointwise in $a$)
  \begin{align*}
    \int_{x \in [\delta,\delta]} \int_{a \in \R_+}
    a^{\gamma-1} v^2 \dxa \le  2
    \left\| \var^{1/2} a^{(\gamma-1)/2} \partial_x v
    \right\|^2_{L^2_{x,a}} +  \left\| \var^{-1/2}
    a^{(\gamma-1)/2} v \right\|^2_{L^2_{x,a}}.
  \end{align*}
  We set $\var = \eta a^{(\gamma-1)}$ and, for $\gamma \ge 3/2$
  we have $2(\gamma-1) \ge m+1$ so 
  \begin{align*}
    \int_{x \in [\delta,\delta]} \int_{a \in \R_+}
    a^{\gamma-1} v^2 \dxa \le 2 \delta \eta
    \left\| a^{1/2} \partial_x v \right\|^2_{L^2_{x,a}} 
    + \frac{2 \delta}{\eta} \left\| v \right\|^2_{L^2_{x,a}}
  \end{align*}
  and finally, with $\eta := \gamma^{-1}$,
  \begin{align*}
    &\frac{\D}{\dt} \int_\R \int_{\R_+} |u|^p \dxa 
    \\
    &\le
     (p-1) ( \eta \gamma  - 1 ) \int_\R \int_{\R_+} a
    \left|\partial_x \left( u^{\frac{p}2} \right) \right|^2 \dxa
      + \frac{\gamma (p-1)}{\eta} \int_\R \int_{\R_+}
      |u|^p \dxa \\
    & \le \gamma^2 (p-1) \int_\R \int_{\R_+} |u|^p \dxa
  \end{align*}
  which proves the estimate by the Gronwall lemma.





	
  To construct solutions, one can for instance use the following
  iterative scheme, for some $\epsilon>0$:
  \begin{equation}
    \label{Pacman_weak_prototip_0}
    \begin{cases}
      \partial_{t}\tilde{u}-(\epsilon+\varphi\left( a\right)
      )\partial _{xx}\tilde{u}+\partial_{a}\left( \chi\left(
          a\right) \psi\left( x\right)
        \tilde{u}\right)  =0, & x \in \R, a\ge 0, t>0,\\[2mm]
      \tilde{u}_{|t=0}=u_{0}, & x \in \R, a\ge 0,
    \end{cases}
  \end{equation}
  where $\varphi$, $\chi$ are defined by
  \begin{equation}
    \label{def_phi_chi}
    \varphi\left(  a\right)  =
    \begin{cases}
      0&\text{ if }a\leq0,\\
      a&\text{ if }a\in\lbrack0,A_{0}],\\
      0&\text{ if }a\geq2A_{0},
    \end{cases}
    \quad \textrm{ and } \quad 
    \chi\left(  a\right)  =
    \begin{cases}		
      0&\text{ if }a\leq0,\\
      f(a)&\text{ if }a\in\lbrack0,A_{0}],\\
      f(A_{0})&\text{ if }a\geq A_{0}.
    \end{cases}
  \end{equation}
  Note that
  $\chi^{\prime}\left( a\right) \in L^{1}\left( \mathbb{R}\right)
  \cap L^{\infty}\left( \mathbb{R}\right) $ with the weak
  derivative
  \begin{equation}
    \label{chi_prime}
    |\chi^{\prime}\left(  a\right)|  =
    \begin{cases}
      0&\text{ if }a\leq0,\\
      \gamma a^{\gamma-1}&\text{ if }a\in(0,A_{0}),\\
      0&\text{ if }a\geq A_{0}.
    \end{cases}
  \end{equation}
  We then further approximate by the standard Galerkin
  method~\cite{Book_BCD}:
  \begin{equation}
    \label{Pacman_weak_prototip}
    \begin{cases}
      \partial_{t}\tilde{u}-\mathbb P_{n}\left\{
        (\varepsilon+\varphi\left( a\right)
        )\partial_{xx}\tilde{u}+\partial_{a}\left( \chi\left(
            a\right) \psi\left( x\right) \tilde{u}\right)
      \right\} =0,\\[2mm]
      \tilde{u}_{|t=0}= \mathbb P _{n}u_{0},
    \end{cases}
  \end{equation}
  with 
  \begin{equation*}
    \mathbb P_{n}\tilde u:=\mathcal{F}^{-1}\left( \mathds{1}_{\left[
          -n,n\right] \times\left[ -n,n\right] }\left( x,a\right)
      \mathcal{F} \tilde u\right)
  \end{equation*}
  and $\mathcal{F}$ is the Fourier transform in $(x,a)$. We
  use, as usual in these methods, Bernstein's lemma
  see~\cite[Lemma~2.1, p.69]{Book_BCD}-\cite[p.191]{Book_BCD} to
  control derivative thanks to the Fourier restriction, and the
  Cauchy-Lipschitz theorem theorem shows global
  existence-uniqueness in 
  \begin{equation*}
    \mathbf{E}_{n}:=\left\{ \tilde{u}\in L^{2}\left(
        \mathbb{R}^{2}\right)
      :\operatorname*{Supp}\mathcal{F}\tilde {u}\subset\left[
        -n,n\right] \times\left[ -n,n\right] \right\}.
  \end{equation*}
  for each $n \ge 1$. It is then straightforward to check that
  our previous a priori estimates are satisfied, uniformly in $n
  \ge 1$, for this approximate problem, and therefore allow to
  pass to the limit and construct a solution to the limit
  problem.
\end{proof}

\subsection{The case of deceleration with
  $\gamma\in\left( 0,1\right)$}
\label{Section_blowup}


We study the local well-posedness for strong solution with a
support condition, then for weak solutions with a support
condition, and finally we prove finite-time singularity and
ill-posedness for weak solutions in $L^1 \cap L^p$.
\begin{prop}
  \label{Prop_regular_solution}
  Consider $\gamma \in (0,1)$,
  $\psi\in C_{c}^{\infty}\left( -1,1\right) $ with\footnote{This
    assumption is not crucial, but it simplifies some technical
    aspects of the proof.}  $0\leq \psi\left( x\right) \leq1$,
  and $0 \le u_0 \in C^\infty_c$ with
  \begin{equation*}
    \operatorname*{Supp}u_{0} \subset\mathbb{R}\times\lbrack a_{0},A_{0}],
  \end{equation*}
  for some $0<a_0<A_0$. Then there exists a time
  $T \in (0,(1-\gamma)^{-1} a_0 ^\gamma)$ and a unique non-negative
  mass-preserving smooth solution to~\eqref{Pacman_prototip}
  with
  $\operatorname*{Supp}u (t,\cdot,\cdot) \subset \R \times [
  a(t),A_{0}]$ and
  \begin{equation*}
    a(t) := \left( a_0^{\gamma}-\left( 1-\gamma\right)
      t\right)^{\frac{1}{1-\gamma}}.
  \end{equation*}
\end{prop}

\begin{proof}[Proof of Proposition~\ref{Prop_regular_solution}]
  As long as the support stays away from $a=0$, the construction
  of the smooth non-negative mass-preserving is done by standard
  approximation proceedures as before. Let us give the a priori
  estimate on the support. 
  Let us compute
  \begin{align*}
    \frac{\D}{\dt}\int_{\mathbb{R}}\int_{a(t)}^{+\infty}
    u(t,x,a) \dxa 
    & =\int_{\mathbb{R}}\int_{a(t)}^{+\infty}\partial_{t} u(t,x,a)
      \dxa + \int_{\mathbb{R}} u (t,x, a(t)) a(t)^{\gamma} \dx \\
    &  =\int_{\mathbb{R}}\int_{a(t)}^{+\infty} \Big[
      a \partial_{xx}u+\partial_{a}(\psi u) \Big] \dxa +
      \int_{\mathbb{R}} u(t,x,a(t)) a(t)^{\gamma} \dx \\
    &  =\int_{\mathbb{R}} u(t,x,a(t)) \Big[ a(t)^{\gamma}
      -\psi (x)  \Big] \dx \geq 0,
\end{align*}
since $u=0$ at $a(t)$. The conservation of mass and
non-negativity of $u$ then implies 
\begin{equation*}
  \int_{\R}\int_{a(t)} ^{+\infty} u
  (t,x,a)\dxa=\int_{\R}\int_{\R_+} u(t,x,a) \dxa,
\end{equation*}
and finally 
\begin{equation*}
  \operatorname*{Supp} u (t,\cdot,\cdot) \subset \R \times \left[
    a(t),+\infty\right].
\end{equation*}
This ends the proof of Proposition~\ref{Prop_regular_solution}.
\end{proof}

As an immediate consequence of
Proposition~\ref{Prop_regular_solution} we have the following
\begin{prop}
  \label{existence}
  Consider $u_{0}\in L^{1} \cap L^p ( \R \times \R_+ )$, for
  $p>1$, such that
  \begin{equation*}
    \operatorname*{Supp}u_{0} \subset\mathbb{R}\times\lbrack
    a_{0},A_{0}]\text{ and }u_{0}\geq0.
  \end{equation*}
  Then, there exists a $T>0$ and a unique weak solution
  $0 \le u \in L^1 \cap L^p$ to~\eqref{Pacman_prototip} on
  $\left( 0,T\right) $ so that
  \begin{equation*}
    \forall \, t\in\left( 0,T\right), \quad
    \operatorname*{Supp}u (t, \cdot, \cdot)
    \subset\mathbb{R}\times \left[ \frac{a_{0}}{2},A_{0} \right].
  \end{equation*}	
\end{prop}

\begin{proof}[Proof of Proposition~\ref{existence}]
  The proof follows by regularization of the initial data and the
  coefficients. When the support is bounded away from $a=0$, the
  $L^p$ estimate  
  \begin{align*}
    \frac{\D}{\dt}\int_\R \int_{\R_+}|u|^p \dxa \leq
    & - p(p-1) \int_\R \int_{\R_+}
      a\left\vert \partial_{x}u\right\vert ^{2} |u|^{p-2} \dxa \\
    & + \gamma(p-1) \int_\R \int_{\R_+}
      a^{\gamma-1}\psi\left( x\right) |u|^p\dxa.
  \end{align*}
  can be closed since the last term is bounded in terms of the
  left hand side.
\end{proof}

\begin{proof}[Proof of Theorem~\ref{Blow-up_viriel} for the case (c)]
  When $\gamma \in (1/2,1)$, we use
  $\left( 1-x^{2}\right) ^{2}\mathds{1}_{[-1,1]}\left( x\right)
  a^{-m}$ as a test function in the weak formulation. In the case
  $\gamma \in (0,1/2)$, we use instead
  $\frac{(1-x^{2})^{2}}{\left(
      \epsilon+a\right)^{m}}\mathds{1}_{[-1,1]}(x)$ and pass to
  the limit $\epsilon \to 0$. In both cases, we obtain
  \begin{align}
    \label{estimate1}
    &  \int_{0}^{\infty}\int_{-1}^{1}\frac{\left(  1-x^{2}\right)  ^{2}%
      u(t,x,a)}{a^{m}}\dxa-\int_{0}^{\infty}\int_{-1}^{1}\frac{\left(
      1-x^{2}\right)  ^{2}u_0(x,a)}{a^{m}}\dxa\nonumber\\
    &  =\int_{0}^{t}\int_{0}^{\infty}\int_{-1}^{1}4\left(  x^{2}-3\right)
      a^{1-m}u\left(  t,x,a\right)  \dxa \nonumber \\
    &\qquad +\frac{m}{2}\int_{0}^{t}\int_{0}^{\infty
      }\int_{-1}^{1}\frac{\left(  1-x^{2}\right)
      ^{2}u(t,x,a)}{a^{m+1-\gamma}}\dxa.
  \end{align}
  Owing to the fact that
  $u\in C([0,T);L_{\mathrm{weak}}^{2}(\mathbb{R}^{2})),$ the application
  \begin{equation}
    \label{y_def}
    t \mapsto y(t) := \int_{0}^{\infty}\int
    _{-1}^{1}\frac{\left(  1-x^{2}\right)  ^{2}u(t,x,a)}{a^{m}}\dxa
  \end{equation}
  is continuous. The support condition on $u$ implies
  \begin{equation}
    \label{estimate3}
    \int_{0}^{\infty}\int_{-1}^{1}4\left(  x^{2}-3\right)  a^{1-m}u\left(
      t,x,a\right)  \dxa\geq-12A_{0}^{1-m} \mathcal M(u_{0}). 
  \end{equation}
  Next, the H\"{o}lder inequality implies
  \begin{align*}
    &  \int_{0}^{\infty}\int_{-1}^{1}\frac{\left(  1-x^{2}\right)  ^{2}%
      u(t,x,a)}{a^{m}}\dxa \\
    &  \leq\left(  \int_{0}^{\infty}\int_{-1}^{1}\frac{\left(  1-x^{2}\right)
      ^{2}u(t,x,a)}{a^{m+1-\gamma}}\dxa\right)  ^{\frac{m}{m+1-\gamma}}
     \left(
      \int_{0}^{\infty}\int_{-1}^{1}\left(  1-x^{2}\right)  ^{2}u\left(
      t,x,a\right)  \dxa\right)  ^{\frac{1-\gamma}{m+1-\gamma}}\\
    &  \leq\left(  \int_{0}^{\infty}\int_{-1}^{1}\frac{\left(  1-x^{2}\right)
      ^{2}u(t,x,a)}{a^{m+1-\gamma}}\dxa\right)  ^{\frac{m}{m+1-\gamma}}
      \mathcal M(u_{0})^{\frac{1-\gamma}{m+1-\gamma}}.
   \end{align*}
  Combined with~\eqref{estimate3} and~\eqref{estimate1} it
  implies the following differential inequality
  \begin{equation}
    \label{y equation}
    y\left(  t\right)
    \geq y\left(  0\right)  -12A_0 ^{1-m}t \mathcal M(u_{0})
    +\frac{m}{2} \mathcal
    M(u_{0})^{-\frac{1-\gamma}{m}}\int_{0}^{t}y\left(
      \tau\right)^{\frac{m+1-\gamma}{m}} \dtau.
  \end{equation}

  We use the following elementary Gronwall-type lemma:
  \begin{lemm}
    \label{blow-up}
    Let $C_1, C_2, \theta>0$ and $y(t)$ be a continuous function
    on $[0,T)$ such that
    \begin{equation*}
      y\left( t\right) \geq y\left( 0\right) +\int_{0}^{t}\left(
        C_{2} y^{1+\theta}\left( \tau\right) - C_{1}\right) \dtau.
    \end{equation*}
    Then if
    $\alpha:= C_{2}y^{1+\theta}(0) -C_{1} > 0$ the
    time of existence is bounded by
    \begin{equation*}
      T \le T(\alpha) := \alpha^{-\frac{\theta}{1+\theta}}
      \frac{1}{\theta} \left( \frac{C_1}{C_2}
      \right)^{\frac{1}{1+\theta}}
    \end{equation*}
    and $y(t)$ blows-up at $t \to T(\alpha)$ like
    \begin{equation*}
      y(t) \ge y(0) \left( \frac{1}{1 -  t \frac{\alpha 
            \theta}{y(0)}} \right)^{\frac{1}{\theta}}.
    \end{equation*}
\end{lemm}

The proof of this lemma is straightforward by observing that
the bound from below $C_{2}y^{1+\theta}(t)) -C_{1} \ge \alpha >0$
is propagated in time, and thus $y(t) \ge y(0)$ on the time of
existence and finally $y(t)^{-1-\theta} y'(t) \ge
\alpha y(0)^{-1-\theta}$.

The proof of Theorem~\ref{Blow-up_viriel} (case (c)) follows from
applying Lemma~\ref{blow-up} to~\eqref{y equation} with
\begin{equation*}
  C_1 := 12 A_0^{1-m} \mathcal M(u_0), \quad C_2 := \frac{m}{2}
  \mathcal M(u_0)^{-\frac{1-\gamma}{m}}, \quad \theta := \frac{1-\gamma}{m}.
\end{equation*}
We only need to construct initial data $u_{0}$ such that $C_2
y(0)^{1+\theta} - C_1 \ge \alpha >0$. We write
\begin{equation*}
  \alpha = \tilde \alpha A_0^{1-m} \mathcal M(u_0)
\end{equation*}
and we want to construct $u_0$ so that 
\begin{equation*}
  \frac{m}{2}  \left(  \int_{0}^{\infty}\int
    _{-1}^{1}\frac{\left(  1-x^{2}\right)
      ^{2}u_0(x,a)}{a^{m}}\dxa \right)^{\frac{m+1-\gamma}{m}}
  \geq (12 + \tilde \alpha) A_{0}^{1-m} \left(
    \int_{0}^{\infty}\int_\R  u_0(x,a) \dxa \right)^{\frac{m+1-\gamma}{m}},
\end{equation*}
which is equivalent to
\begin{equation*}
  \int_{0}^{\infty}\int_{-1}^{1}\frac{\left( 1-x^{2}\right)
    ^{2}u_{0}(x,a)}{a^{m}}\dxa\geq\left(
    \frac{(12+\tilde \alpha) A^{1-m}}{m}\right)^{\frac{m}
    {m+1-\gamma}}\int_{0}^{\infty}\int_{\mathbb{R}}u_{0}\left(x,a\right)
  \dxa.
\end{equation*}
This is clear since the integrand in the left hand side has a
singularity $a^{-m}$ in the variable $a$. Moreover this is an
homogeneous equation in $u$ so the size of $u$ does not
matter. By concentrating the initial data near zero at a distance
$a_0$ with $a_0$ small, we can satisfy the inequality with
$\tilde \alpha$ as large as wanted, and therefore with a time of
explosion as small as wanted. This ends the proof of
Theorem~\ref{Blow-up_viriel}.
\end{proof}

\subsection{A remark on the blow-up times}

Consider $u_0$ smooth and a corresponding weak solution
to~\eqref{Pacman_prototip} that blows-up at a $T^*$. The
existence time of a classical solution to~\eqref{Pacman_prototip}
is then strictly shorter than $T^*$.
Indeed,
if the solution remains regular on $(0,T^*)$ then
\begin{equation*}
  \forall \, t \in [0,T^*), \quad
  \lim_{a\rightarrow0}a^{\gamma}u\left(  t,x,a\right)  =0.
\end{equation*}
But then, we can use the maximum principle on the equation
satisfied by $a^\gamma u$ to get
\begin{equation}
  \label{Maximum_principle1}
  \forall \, t \in [0,T^*), \quad
  \sup_{x,a}a^{\gamma}u\left(  t,x,a\right)  \leq\sup_{x,a}a^{\gamma}
  u_{0}\left(  x,a\right).
\end{equation}
The Theorem~\ref{Blow-up_viriel} (case (c)) for
$m\in (0,1-\gamma )$ is in contradiction
with~\eqref{Maximum_principle1} since then
\begin{equation*}
  \int_{0}^{\infty}\int_{-1}^{1}\frac{\left( 1-x^{2}\right) ^{2}
    u(t,x,a)}{a^{m}}\dxa \le \left( \sup_{x,a} a^\gamma
    u_{0}(x,a) \right) 
  \left( \int_{-1}^{1}\left( 1-x^{2}\right) 
    ^{2}dx \right) \left( \int_{0}^{A_{0}}\frac{da}{a^{m+\gamma}}
  \right)
\end{equation*}
which remains finite. Note moreover than since the time of
explosion can be as small as wanted, we easily deduced that the
equation is ill-posed both for regular and weak solutions, when
one removes the support condition.

\section{Variants and extensions}
\label{sec:autre_models}


\subsection{Nonlinear diffusion}

The argument we used above is robust enough to deal with
nonlinear diffusions. Consider, for $q\in [1,1+1/\gamma)$,
the equation
\begin{equation}
  \label{millieux_poreux}
  \begin{cases}
    \partial_{t}u=a\partial_{xx}^{2}(u^q)+\psi\left( x\right)
    \partial_{a}(a^{\gamma}u) ,
    & x \in \R, a\ge 0, t>0,\\[2mm]
    u_{t=0}=u_0 , & x \in \R, a\ge 0,
  \end{cases}
\end{equation}
One can extend Theorem~\ref{Blow-up_viriel} to the weak $L^1 \cap
L^p$ solutions to~\eqref{millieux_poreux}. 
The proof is similar. The only difference is in estimating the
first term in~\eqref{estimate1}, which now becomes
\begin{equation}
  \label{term_millieux}
  \int_{0}^{t} \int_{-1}^{1} \int_{\R_+} 4\left(  x^{2}-3\right)
  a^{1-m}u^q (t,x,a) \dtau \dxa.
\end{equation}
To control it we use an additional a priori estimate:
integrate~\eqref{millieux_poreux} against
$p\left( a^{\gamma}u\right) ^{p-1}$ to bound 
\begin{equation*}
  \sup_{t \ge 0}\int
  _{0}^{\infty}\int_{-1}^{1}a^{(p-1)\gamma}u^q \left(
    t,x,a\right) \dxa.
\end{equation*}
This last quantity controls~\eqref{term_millieux} if $m$ is
chosen such that $m \in (0,1-\left( q-1\right) \gamma)$ (the
interval is not empty by assumption). Then one constructs $u_{0}$
such that\\
$\int_{0}^{\infty }\int_{-1}^{1}\frac{\left( 1-x^{2}\right)
  ^{2}u(t,x,a)}{a^{m}}\dxa$ blows-up in finite time as above. 

\subsection{Logistic memory effect}

Consider, for some $\lambda >0$ and $A^*>0$ and
$\gamma \in (0,1)$, the equation
\begin{equation}
  \label{memory}
  \begin{cases}
    \partial_{t}u = a\partial_{xx}u - \lambda \partial_{a}\left(
      g\left( a\right) u\right)  +\psi\left(  x\right)
    \partial_{a}\left(  a^{\gamma}u\right), & x \in \R, a\ge 0, t>0,\\[2mm]
    u_{|t=0}=u_{0}, & x \in \R, a\ge 0,
  \end{cases}
\end{equation}
where $\psi$ is as before, localized in $[-1,1]$ and with mass $1$, and 
\begin{equation}
  \label{def:g}
  g\left( a\right) =
  \begin{cases}
    0&\text{ if }a\leq0, \\
    a(A^{\ast}-a)&\text{ if }a\in\lbrack0,A^{\ast}], \\
    0&\text{ if }a\geq A^{\ast}.
  \end{cases}
\end{equation}
We consider $\gamma \in (0,1)$, and initial data whose support is
included in $\R \times (0,A^{\ast})$. It is easy to check that
this support condition is propagated by the equation. Define the
function 
\begin{equation}
  \label{def:h}
  h:[0,A^{\ast}]\rightarrow \mathbb{R}, \quad a \mapsto \lambda (A^{\ast}-a)a-a^{\gamma}.
\end{equation}

\begin{prop}
  If $h$ has a root $a^*$ in $(0,A^*)$, and if the initial data
  satisfies
  \begin{equation*}
    \operatorname*{Supp}u_{0}\subset [a^*,A^{\ast}]\times\mathbb{R}
  \end{equation*}
  then this support condition is preserved and the unique weak
  $L^1 \cap L^p$ solution to~\eqref{memory} with~\eqref{def:g}
  ($p>1$) exists globally in time.

  But in any case, for $a_0>0$ small enough,  
  there are arbitrarily small non-negative initial data
  \begin{equation*}
    u_{0} \in L^{1} \cap L^p (\R \times \R_+),\quad
    \operatorname*{Supp}u_{0} \subset\mathbb{R}\times [a_{0},A_{0}],
  \end{equation*}
  so that the corresponding unique weak solution
  $0 \le u \in L^1 \cap L^p$ to~\eqref{memory}-\eqref{def:g}
  blows-up in finite time $T^{\ast} \in (0,+\infty)$, in the
  sense
  \begin{equation}
    \label{accumulation}
    \limsup_{t\rightarrow T^{\ast}}\left\Vert u\left(  t\right)
    \right\Vert_{L^p(\R\times \R_+)}=+\infty.
  \end{equation}
  Moreover this time of explosion can be as small as wanted
  provided that $a_0$ is small enough.
\end{prop}

Note that the existence of a positive root $a^* \in (0,A^*)$
always exist for $\lambda$ large enough. So large memory erasing
effects can prevent the blow-up mechanism, as expected.

\begin{proof}
  The construction of weak solutions is done as before. To prove
  blow-up, we argue as in Theorem~\ref{Blow-up_viriel} and
  integrate the equation against
  $\left( 1-x^{2}\right) ^{2}\mathds{1}_{[-1,1]}\left( x\right)
  a^{-m}$ to get
  \begin{align*}
    &  \int_{-1}^{1} \int_{\R_+} \frac{\left(  1-x^{2}\right)  ^{2}
      u(t,x,a)}{a^{m}}\dxa - \int_{-1}^{1} \int_{\R_+} \frac{\left(
      1-x^{2}\right)  ^{2}u_{0}(x,a)}{a^{m}}\dxa \\
    &  =\int_{0}^{t} \int_{-1}^{1} \int_{\R_+} 4\left(  x^{2}-3\right)
      a^{1-m}u\left(  s,x,a\right) \ds \dxa  \\
    &  \qquad
      +\frac{m}{2}\int_{0}^{t} \int_{-1}^{1} \int_{\R_+} \frac{\left(
      1-x^{2}\right)  ^{2}u(s,x,a)}{a^{m+1-\gamma}} \ds \dxa \\
    &\qquad -m\int_{0}^{t} \int_{-1}^{1} \int_{\R_+}
      \frac{\lambda (A^{\ast}-a)\left(
      1-x^{2}\right)^{2}u(s,x,a)}{a^{m}} \ds \dxa.
  \end{align*}
  Then the function
  $y(t) := \int _{-1}^{1} \int_{\R_+} \frac{\left(
      1-x^{2}\right) ^{2}u(t,x,a)}{a^{m}}\dxa$ defined as before
  satisfies
  \begin{equation*}
    y \left( t\right) \geq y\left( 0\right) +\int_{0}^{t}\left(
      C_{2} y^{1+\theta}\left( \tau\right) -C_{1}-C_{3}y\left(
        \tau\right) \right) \dtau,
  \end{equation*}
  with $C_{1}$ and $C_{2}$ as in the case $\lambda =0$. Using
  again Young's inequality we deduce 
  \begin{equation*}
    y\left( t\right) \geq y\left( 0\right) +\int_{0}^{t}\left(
      \tilde{C}_{2}y^{1+\theta}\left( \tau\right)
      -2\tilde{C}_{1}\right) \dtau,
  \end{equation*}
  for $\tilde{C}_{1}\tilde{C}_{2}>0$ depending on $C_{1},C_{2}$
  and $C_{3}$. The blow-up argument can then be applied as
  before.

  To prove that when a positive root $a^* \in (0,A^*)$ of $h$
  exists, one can propagate a support condition in $(a^*,A^*)$,
  use the following a priori (for smooth solutions)
  \begin{align*}
    & \frac{\D}{\dt}\int_{\mathbb{R}}\int_{a^*}^{A^{\ast}}u\left(
      t,x,a\right) \dxa \\  
    & =-\int_{\mathbb{R}}\int_{a^*}^{A^{\ast}}\psi\left(  x\right)
      \partial_{a}\left(  h\left(  a\right)  u\right)
      \dxa-\int_{\mathbb{R}} \int_{a^*}^{A^{\ast}}(1-\psi\left(
      x\right) )\partial_{a}\left(  \lambda a\left( A^{\ast}-a\right)
      u\right) \dxa \\ 
    & =-\int_{\mathbb{R}}\int_{a^*}^{A^{\ast}}\psi\left(
      x\right)  \partial_{a}\left(  h\left(  a\right)  u\right)
      \dxa-\int_{\mathbb{R}}\int_{a^*}^{A^{\ast}}(1-\psi\left(
      x\right)  )\partial_{a}\left(  \lambda a\left(
      A^{\ast}-a\right)  u\right)  \dxa \\
    & =\int_{\mathbb{R}}(1-\psi\left(  x\right)  )\lambda
      a^*(A^{\ast}-a^*)u\left(t,x,a^*\right) \dx \geq0.
  \end{align*}
  Thus
  \begin{equation*}
    \int_{\mathbb{R}}\int_{a^*}^{A^{\ast}}u\left( t,x,a\right)
    \dxa\geq \int_{\mathbb{R}}\int_{a^*}^{A^{\ast}}u_{0}\left(
      x,a\right) \dxa=\int_{\mathbb{R}}\int_{0}^{+\infty}u\left(
      t,x,a\right) \dxa,
  \end{equation*}
  which, given the non negativity of $u$ implies that the support
  of $u$ is included in $\R \times \left[ a^*,A^{\ast}\right]
  $. It is then easy to deduce that the solution is global.
\end{proof}

\subsection{The two-dimensional case}
\label{sec:2d}

We consider the equation
\begin{equation}
  \label{eq:boule_2d}
  \begin{cases} 
  \partial_{t}u-a\Delta_x u=\mathds{1}_{B\left(0,1\right)
  }\left(  x\right) \partial_{a}\left(  a^{\gamma}u\right), & x \in \R^2, a\ge 0, t>0,\\[2mm]
    u_{|t=0}=u_{0}, & x \in \R^2, a\ge 0.
  \end{cases}
\end{equation}
We focus on the proof of the occurrence of a blow-up. Indeed, the
proof of the existence and uniqueness of the solution is similar
to the case of dimension $1$ in Section~\ref{sec:interval}.
\begin{theo}
  \label{prop:explosion_2d}
  \bigskip Assume that $\gamma\in\left( 0,1\right) $. For any
  $m\in\left( 0,\gamma\right) $ there exists arbitrarily small
  non-negative, compactly supported initial data
  $0 \le u_{0} \in L^1 \cap L^p(\R^2 \times \R_+)$ so that the
  corresponding unique weak $L^1 \cap L^p$ solution
  to~\eqref{eq:boule_2d} blows-up in finite time $T^* >0$,
  i.e.
  \begin{equation}
    \label{accumulation_2d}
    \lim_{t\rightarrow T^*}\int_{B(0,1)} \int_{\R_+}
    \frac{\left(1-x^{2}\right)^{2}u(t,x,a)}{a^{m}} \dxa=\infty
    \quad \text{ and } \quad \limsup_{t\rightarrow T^{\ast}}\left\Vert u\left(  t\right)
    \right\Vert_{L^p(\R^2\times \R_+)}=+\infty.
  \end{equation}
\end{theo}

\begin{proof}[Proof of Theorem~\ref{prop:explosion_2d}]
  Consider $m\in(0,\gamma)$ and calculate
  \begin{align}
     \frac{\D }{\dt} \int_{B(0,1)} \int_{\R_+}
    \frac{\left(1-\left\vert x\right\vert ^{2}\right)
        ^{2}}{a^{m}}u\left(  t,x,a\right) \dxa 
      = & \int_{B(0,1)} \int_{\R_+} \frac{\left(  1-\left\vert x\right\vert ^{2}\right)
       ^{2}}{a^{m}}a\Delta_x u\left( t,x,a\right)
       \dxa  \nonumber \\
    & + \int_{B(0,1)} \int_{\R_+} \frac{\left(
        1-\left\vert x\right\vert ^{2}\right)
        ^{2}}{a^{m}}\partial_{a}\left( a^{\gamma}u\right) \dxa. \label{der:energy} 
  \end{align}
	
  Regarding the last term of the right hand side, by
  nonnegativity of $u$, we have
  \begin{align*}
    \int_{B(0,1)} \int_{\R_+} \frac{\left(
    1-\left\vert x\right\vert ^{2}\right)
    ^{2}}{a^{m}}\partial_{a}\left( a^{\gamma}u\right) \dxa 
    = & \left[ \int_{B(0,1)}\frac{\left(  1
        -\left\vert x\right\vert ^{2}\right)
        ^{2}}{a^{m}}a^{\gamma}u(t,x,a)\dx\right]
        _{a=0}^{a=+\infty} \\
      &  - \int_{B(0,1)} \int_{\R_+} (-m)\frac{\left(  1-\left\vert x\right\vert ^{2}\right)
        ^{2}}{a^{m+1}}a^{\gamma}u\left(  t,x,a\right) \dxa \\
    \geq & m \int_{B(0,1)} \int_{\R_+}
           \frac{\left( 1-\left\vert x\right\vert ^{2}\right)
           ^{2}}{a^{m+1-\gamma}}u\left( t,x,a\right)  \dxa .
  \end{align*}
	
  Regarding the first term of the right hand side
  of~\eqref{der:energy}, we integrate by parts to obtain
  \begin{align*}
    \int_{B(0,1)} \int_{\R_+} \frac{\left(
    1-\left\vert x\right\vert ^{2}\right)
    ^{2}}{a^{m}}a\Delta_x u \dxa 
    = & \int_{\partial B(0,1)} \int_{\R_+} \frac{\left(
        1-\left\vert x\right\vert ^{2}\right)
        ^{2}}{a^{m}}a\nabla_x u
        \frac{x}{\left\vert x\right\vert }\dxa \\
      & - \int_{B(0,1)} \int_{\R_+} a^{1-m}\nabla_x \left(  1-\left\vert
        x\right\vert^{2}\right)  ^{2}\nabla_x u \dxa. 
  \end{align*}
	
  Using that
  \begin{equation*}
    \nabla_x\left( 1-\left\vert x\right\vert ^{2}\right)
    ^{2}=2\left( \left\vert x\right\vert ^{2}-1\right)
    \nabla_x\left\vert x\right\vert ^{2}=4\left( \left\vert
        x\right\vert ^{2}-1\right) x,
  \end{equation*}
  we deduce
  \begin{align*}
    \int_{B(0,1)} \int_{\R_+}
    a^{1-m}\nabla\left(  1-\left\vert x\right\vert ^{2}\right)
      ^{2}\nabla_x u  \dxa 
    = & 4 \int_{\partial B(0,1)} \int_{\R_+} a^{1-m} u  \left(
      \left\vert x\right\vert ^{2}-1\right)  \left\vert
      x\right\vert \\
      &  - \int_{B(0,1)} \int_{\R_+}
        a^{1-m} \Delta_x \left(  \left(  1-\left\vert x\right\vert
      ^{2}\right)  ^{2}\right) u  \dxa .
  \end{align*}
	
  Moreover, since
  \begin{equation*}
    \Delta_x\left( \left(  1-\left\vert x\right\vert ^{2}\right)
      ^{2}\right) = 4 \nabla_x \cdot \left[ \left(  \left\vert
        x\right\vert ^{2}-1\right) x \right] =8\left(  2\left\vert
        x\right\vert ^{2}-1\right),
  \end{equation*}
  we get
  \begin{align*}
    \int_{B(0,1)} \int_{\R_+} \frac{\left(
      1-\left\vert x\right\vert ^{2}\right)
    ^{2}}{a^{m}}a\Delta_x u \dxa
    &  = \int_{B(0,1)} \int_{\R_+}
      a^{1-m}\Delta_x \left\{  \left( 1-\left\vert x\right\vert
      ^{2}\right)  ^{2}\right\}  u  \dxa \\
    & = 8 \int_{B(0,1)} \int_{\R_+} a^{1-m}\left(  2\left\vert
      x\right\vert ^{2}-1\right) u \dxa \\
    & \ge - 8 \int_{B(0,1)} \int_{\R_+} a^{1-m} u \dxa \\
    &  \ge - 8 A_{0}^{1-m} \int_{\mathbb{R}^{2}} \int_{\R_+} u_{0} \dxa
  \end{align*}
  and the rest of the argument is similar to the case of
  dimension $1$. 
\end{proof}

\end{document}